\documentclass[12pt,reqno]{article}

\usepackage[usenames]{color}
\usepackage{amssymb}
\usepackage{amsmath}
\usepackage{amsthm}
\usepackage{amsfonts}
\usepackage{amscd}
\usepackage{graphicx}
\usepackage{lscape}

\usepackage[colorlinks=true,
linkcolor=webgreen,
filecolor=webbrown,
citecolor=webgreen]{hyperref}

\definecolor{webgreen}{rgb}{0,.5,0}
\definecolor{webbrown}{rgb}{.6,0,0}

\usepackage{color}
\usepackage{fullpage}
\usepackage{float}

\usepackage{graphics}
\usepackage{latexsym}
\usepackage{epsf}
\usepackage{breakurl}

\def\suchthat{\, : \, }

\def\Enn{\mathbb{N}}

\newenvironment{smallarray}[1]
{\null\,\vcenter\bgroup\scriptsize
\arraycolsep=.13885em
\hbox\bgroup$\array{@{}#1@{}}}
{\endarray$\egroup\egroup\,\null}

\begin{document}

\theoremstyle{plain}
\newtheorem{theorem}{Theorem}
\newtheorem{corollary}[theorem]{Corollary}
\newtheorem{lemma}[theorem]{Lemma}
\newtheorem{proposition}[theorem]{Proposition}

\theoremstyle{definition}
\newtheorem{definition}[theorem]{Definition}
\newtheorem{example}[theorem]{Example}
\newtheorem{conjecture}[theorem]{Conjecture}

\theoremstyle{remark}
\newtheorem{remark}[theorem]{Remark}

\title{Counterexamples to a Conjecture of Dombi in Additive Number Theory}

\author{
Jason P. Bell \\
Department of Pure Mathematics\\
and \\
Jeffrey Shallit\footnote{Research funded by a grant from NSERC, 2018-04118.} \\
School of Computer Science \\
\ \\
University of Waterloo \\
Waterloo, Ontario  N2L 3G1\\ Canada \\
\href{mailto:jpbell@uwaterloo.ca}{\tt jpbell@uwaterloo.ca} \\
\href{mailto:shallit@uwaterloo.ca}{\tt shallit@uwaterloo.ca}
}

\maketitle

\begin{abstract}
We disprove a 2002 conjecture of Dombi from additive number
theory.   More precisely, we find examples of sets $A \subset \Enn$
with the property that $\Enn \setminus A$ is infinite, but
the sequence 
$n \rightarrow |\{ (a,b,c) \suchthat n=a+b+c \text{ and }  a,b,c \in A \}|$,
counting the number of $3$-compositions using elements of $A$ only,
is strictly increasing.
\end{abstract}

\section{Introduction}
Let $\Enn = \{0,1,2,\ldots \}$ denote the natural numbers.
A $k$-{\it composition} of an integer $n \in \Enn$ is 
a $k$-tuple of natural numbers $(c_1, c_2, \ldots, c_k)$
such that $n = c_1 + c_2 + \cdots + c_k$.  In contrast
with partitions, the order of the summands
matters in a composition.

Let $A \subseteq \Enn$, and define
$ r(k, A, n)$ to be the number of $k$-compositions of $n$
where the summands are chosen from $A$.  The study of the
function $r$ was initiated by Erd\H{o}s and Tur\'an \cite{Erdos&Turan:1941},
who proved that if $A$ is infinite
then the sequence
$(r(2,A,n))_{n \geq 0}$ cannot be eventually constant.  Erd\H{o}s and 
his co-authors returned to the study of $r(k,A,n)$ in 
many papers; for example,
\cite{Erdos:1956,Erdos&Sarkozy:1985,Erdos&Sarkozy:1986,Erdos&Sarkozy&Sos:1986,Erdos&Sarkozy&Sos:1987,Sarkozy:2006} just to name a few.

In a 1985 paper \cite{Erdos&Sarkozy&Sos:1985}, Erd\H{o}s, S\'ark{\"o}zy, and
S{\'o}s proved that $(r(2,A,n))_{n \geq 0}$ is eventually increasing\footnote{By
``increasing'' in this paper we mean ``(not necessarily strictly) increasing".}
if and only if $A$ is co-finite; that is, if $A$ omits only finitely many
integers.  Their original proof was rather complicated and a considerable
simplification was later found by Balasubramanian \cite{Balasubramanian:1987}.

This result on $r(2,A,n)$ prompted Dombi \cite{Dombi:2002} to study
$r(k,A,n)$ for larger $k$.
He conjectured that
there is no set $A \subset \Enn$ with $A$ co-infinite
and $(r(3,A,n))_{n\geq 0}$ eventually increasing.   In this note
we refute Dombi's conjecture by constructing examples
of co-infinite sets $A$
for which $(r(3,A,n))_{n \geq 0}$ is eventually increasing;
even eventually {\it strictly\/} increasing.
We can also give explicit examples where $(r(3,A,n))_{n \geq 0}$ is strictly
increasing right from the start.

A novelty in our approach for constructing explicit
examples is the use of automatic sequences
and the {\tt Walnut} theorem-prover.   For other applications
of these ideas to additive number theory, see
\cite{Bell&Hare&Shallit:2018,Rajasekaran&Shallit&Smith:2020,Bell&Lidbetter&Shallit:2020,Allouche&Shallit:2022}.

\section{A general construction}

Let $A$ be an infinite subset of $\Enn$.  There are two other ways to view $A$, both useful.  We can
consider its characteristic sequence $(a(n))_{n \geq 0}$, where
$a(n)$ is defined
to be $1$ if $n \in A$ and $0$ otherwise.  We can also consider
its associated formal power series $A(X)$, defined
by $\sum_{n \geq 0} a(n) X^n$.  

Then it is easy to see that for integers $k \geq 1$ we have
$$ \sum_{i \geq 0} r(k, A, i) X^i = A(X)^k .$$
Thus, the claim that $(r(k,A,n))_{n \geq 0}$ is eventually
strictly increasing is equivalent
to the claim that
$(1-X) A(X)^k$ has coefficients that are eventually positive.\footnote{To be clear, by ``positive'' we mean strictly greater than $0$.}

It is easy to see that if $(r(k,A,n))_{n \geq 0}$ is strictly
increasing right from the start, then $0 \in A$.   Thus, the
claim that $(r(k,A,n))_{n \geq 0}$ is strictly increasing right
from the start is equivalent to the claim that
$(1-X) A(X)^k$ has all positive coefficients.

Here is our first main result.  We employ the usual asymptotic
notation, where $f = O(g)$ means $f(n) \leq Cg(n)$ for some $C>0$ and
all sufficiently large $n$.   Also $f = \Theta(g)$ means $f = O(g)$ and
$g = O(f)$.   Finally, $f = o(g)$ means $\lim_{n \rightarrow \infty}
f(n)/g(n) = 0$.

\begin{theorem}
Let $k\geq 3$ be an integer.
Let $F \subseteq \Enn$ and assume $0 \not\in F$.
Let $(f(n))_{n \geq 0}$ be its associated
characteristic sequence and $F(X)$ its associated
power series $\sum_{i \geq 0} f(i) X^i$.
Define $f'(n) = \sum_{0 \leq i \leq n} f(i)$.
Suppose $f'(n) = o(n^{\alpha})$ for some $\alpha < (k-2)/k$
and $A = \Enn \setminus F$.   
Then $(r(k,A,n))_{n \geq 0}$ 
is eventually strictly increasing.
\label{main1}
\end{theorem}

\begin{proof}
Let $A(X) = \sum_{i \geq 0} a(i) X^i$.  It suffices to show that
the coefficients of $(1-X) A(X)^k$ are eventually positive. Now
by the binomial theorem we have
\begin{align}
(1-X) A(X)^k &= (1-X) \left( {1 \over {1-X}} - F(X) \right)^k \nonumber \\
&= (1-X) \left( \sum_{0 \leq i \leq k} (-1)^i {k \choose i} {1\over{(1-X)^{k-i}}} F(X)^i \right) \nonumber \\
&= {1 \over {(1-X)^{k-1}}} + \sum_{1 \leq i \leq k-2} (-1)^i {k \choose i} 
{1 \over {(1-X)^{k-i-1}}} F(X)^i   \nonumber \\
&\quad \quad + (-1)^{k-1} k F(X)^{k-1} + (-1)^k (1-X) F(X)^k.  
\label{maineq}
\end{align}
Now, examining the coefficient of $X^n$ of both sides for $n \geq 0$,
we see that
on the left it is $r(k,A,n)-r(k,A,n-1)$.  (Here by convention $r(k,A,-1) = 0$.)
On the other hand, the coefficient of $X^n$ in
\begin{itemize}
\item $1/(1-X)^{k-1}$ is ${{n+k-2} \choose {k-2}} =
\Theta(n^{k-2})$,
\item $ {k \choose i} {1\over{(1-X)^{k-i-1}}} F(X)^i $
is $O(n^{k-i-2 + i\alpha})$ for $1 \leq i \leq k-2$,
\item $ k F(X)^{k-1}$ is $O(n^{(k-1)\alpha})$, 
\item $(1-X) F(X)^k$ is $O(n^{k\alpha})$.
\end{itemize}
It is now easy to see that the $\Theta(n^{k-2})$ term dominates the
remaining terms, since $\alpha < (k-2)/k$ by hypothesis.
Hence the coefficient of $X^n$ on the right-hand side is eventually
positive, and the result is proved.
\end{proof}

\section{Results for automatic sets}

In some cases we can find a good exact formula for the difference sequence
$$d(n) := r(k,A,n) - r(k,A,n-1).$$
Here by ``good'' we mean that one
can compute the $n$'th term in time polynomial in $\log n$.
In particular, this is possible if $A$ is a $b$-automatic set.
We now explain how this can be done.

Recall that a set $A$ is $b$-automatic if its characteristic sequence
$(a(n))_{n \geq 0}$ can be computed by a DFAO (deterministic
finite automaton with output) reading
the base-$b$ representation of $n$ as input.
Then we can use
known enumeration techniques to show that $(d(n))_{n \geq 0}$ is a
$b$-regular sequence
\cite{Allouche&Shallit:1992,Allouche&Shallit:2003,Allouche&Shallit:2003a}
and even compute a linear representation for it.
By a {\it linear representation\/} for a sequence $x(n)$ we mean
a triple $(v, \gamma, w)$ consisting of
a row vector $v$, a column vector $w$, and a matrix-valued morphism
$\gamma$ with domain $\Sigma_b = \{0,1,\ldots, b-1 \}$ such that
$$ x(n) = v \gamma(z) w$$
for all $z \in \Sigma_b^*$ such that $[z]_b = n$.
(Here $[z]_b$ is the integer represented by the string $z$,
represented in base $b$, with most significant digit at the left.)
This gives us our desired good formula.  

More precisely, since we can write a first-order logical formula
specifying that $(x_1, x_2, \ldots, x_k)$ is a $k$-composition of
elements of $A$, we can obtain a linear representation for
$r(k,A,n)$ from the DFAO for $A$.   Then, using a simple construction
based on block matrices, as in \cite[p.~14]{Berstel&Reutenauer:2011}
we can find a linear representation for the
difference $d(n)$.   This can be carried out
using the {\tt Walnut}
software system, originally designed by Hamoon Mousavi \cite{Mousavi:2016},
and available for free at\\
\centerline{\url{https://cs.uwaterloo.ca/~shallit/walnut.html} .}\\
See \cite{Shallit:2022} for more details.

The {\it rank\/} of a linear representation is the dimension of
the vector $v$.
Linear representations can be minimized (and thus we can find
a linear representation of minimum rank) using an algorithm
of Sch\"utzenberger, as described in \cite[Chap.~2]{Berstel&Reutenauer:2011}.
We used an implementation written in {\tt Maple}.

If a $b$-regular sequence takes only finitely many values, then it is
in fact automatic, and an automaton for it can be computed explicitly
in terms of its linear representation (using the so-called
``semigroup trick''; see \cite[\S 4.11]{Shallit:2022}).

Let $F \subseteq \Enn$ be a $b$-automatic set, let
$(f(n))_{n \geq 0}$ be its characteristic sequence, and
define $f'(n) = \sum_{0 \leq i \leq n} f(i)$.
By the classification of automatic sequences in \cite[Lemmas 2.1--2.3]{Bell&Hare&Shallit:2018} we know that either $f'(n) = O((\log n)^c)$ for some
constant $c$, or $f'(n) \geq n^d$ infinitely often, for some constant $d > 0$.
Furthermore, given a DFAO computing $(f(n))_{n \geq 0}$, it is decidable
which of the two alternatives hold.   If $f'(n) = O((\log n)^c)$, then
we know from Theorem~\ref{main1} that $(r(k,A,n))_{n \geq 0}$
is eventually strictly increasing for $k \geq 3$, where $A = \Enn\setminus F$.

\section{An explicit example}
Now we turn to our particular example.
Choose the ``forbidden set'' $F$ to be 
$$\{ 2^{n+2} - 1 \suchthat n \geq 0 \} = \{ 3,7,15, 31, \ldots \},$$
and define $A = \Enn \setminus F$.  By Theorem~\ref{main1}
we know $(r(k,A,n))_{n \geq 0}$ is eventually strictly increasing
for $k \geq 3$.  We now show it is strictly increasing right from the
start, and find an explicit formula for the difference
$r(k,A,n)-r(k,A,n-1)$.

It follows from Eq.~\eqref{maineq} for $k = 3$ that
$$ D(X) = \sum_{i \geq 0} d(i)X^i = \sum_{i \geq 0} (i+1) X^i - 3 (1 + \sum_{i \geq 0} (\lfloor \log_2 (i+1) \rfloor - 1) X^i) + 3 F(X)^2 - (1-X) F(X)^3 .$$
Hence, by considering the coefficient of $X^n$ on both sides, we have
$$d(n) = (n+1)  - 3 \lfloor \log_2 (n+1) \rfloor + 3 + 3 g(n) - 
(h(n) - h(n-1)),$$
for $n \geq 1$
where $g(n) = [X^n]F(X)^2$ and $h(n) = [X^n]F(X)^3$.

Next we show that $e(n) := d(n) - (n+1) + 3 \lfloor \log_2 (n+1) \rfloor$
is $2$-automatic and find an explicit automaton for it.
We start with the following {\tt Walnut} commands, which can
be typed in verbatim into {\tt Walnut} (without the line numbers):
\begin{verbatim}
(1) reg f msd_2 "0*(11)1*":
(2) reg power2 msd_2 "0*10*":
(3) def log2 "$power2(x) & 1<x & x<=n":
(4) def a3n n "n=x+y+z & ~$f(x) & ~$f(y) & ~$f(z)":
(5) def a3n1 n "n=x+y+z+1 & ~$f(x) & ~$f(y) & ~$f(z)":
(6) def np1 n "x<=n":
(7) def log2n1 n "$log2(n+1,x)":
\end{verbatim}

The explanation for these commands is as follows:  
\begin{itemize}
\item[(1)] asserts that $n$ has a base-$2$ expansion consisting
of two $1$'s followed by any number of $1$'s, allowing leading zeros.
Hence $f(n)$ is true iff $n \in F$.

\item[(2)] asserts that ${\tt power2} (n)$ is true iff $n$ is a power of
$2$.

\item[(3)] asserts that $1<x\leq n$ and $x$ is a power of $2$.

\item[(4)] asserts that $n=x+y+z$ and each of $x,y,z$ belongs to
$\Enn \setminus F$.  It computes a linear representation for the
number of such compositions, which turns out to have rank 100.

\item[(5)] asserts that $n-1=x+y+z$ and each of $x,y,z$ belongs to
$\Enn \setminus F$.  It computes a linear representation for the
number of such compositions, which turns out to have rank 81.

\item[(6)] asserts that $x \leq n$.  It computes a linear representation
for the number of such $x$, which is $n+1$.  It turns out to have rank $2$.

\item[(7)] asserts that $1<x\leq n+1$ and $x$ is a power of $2$; it
computes the number of such $x$, which is $\lfloor \log_2(n+1) \rfloor$.
It turns out to have rank $6$.
\end{itemize}

From the linear representations above in parts (4)--(7),
we can then use a simple block matrix construction to compute the linear
representation 
for $e(n)$.  It has a linear representation of rank $100+81+2+6= 189$.
Now we can minimize it, getting a linear representation
$(v,\gamma,w)$ of rank 10,
as follows:
\begin{align*}
v^T &= \left[\begin{smallarray}{c}
1\\
 0\\
 0\\
 0\\
 0\\
 0\\
 0\\
 0\\
 0\\
 0
\end{smallarray}\right] & \quad
\gamma(0) &= {1 \over {14}} \left[\begin{smallarray}{cccccccccc}
  14&  0&  0&  0&  0&  0&  0&  0&  0&  0\\
   0&  0& 14&  0&  0&  0&  0&  0&  0&  0\\
   0&  0&  0&  0& 14&  0&  0&  0&  0&  0\\
   0&  0&  0&  0&  0&  0& 14&  0&  0&  0\\
   0&  0&  0&  0&  0&  0&  0&  0& 14&  0\\
   0&  0& 24&-32&-19& 24& 30&  0& -5& -8\\
   0&  0&-18& 24&-12&-18&  2&  0& 30&  6\\
   0&  0& 30&-26&-36& 30& 20& 14&  6&-24\\
   0&  0&  0&  0&  0&  0&  0&  0& 14&  0\\
   0&  0& 36&-48&-39& 36& 38&  0&  3&-12\\
\end{smallarray} \right] \\
\gamma(1) &= {1 \over {14}} \left[\begin{smallarray}{cccccccccc}
   0& 14&  0&  0&  0&  0&  0&  0&  0&  0\\
   0&  0&  0& 14&  0&  0&  0&  0&  0&  0\\
   0&  0&  0&  0&  0& 14&  0&  0&  0&  0\\
   0&  0&  0&  0&  0&  0&  0& 14&  0&  0\\
   0&  0&  0&  0&  0&  0&  0&  0&  0& 14\\
   0&  0&  0&-14&  0& 14&  0& 14&  0&  0\\
   0&  0& -6& 22&  3&-48& -4& 14&  3& 30\\
   0&  0&  0&  0&  0&  0&  0& 14&  0&  0\\
   0&  0&-12& 16&  6&-12& -8&  0&  6& 18\\
   0&  0&  6&-22& -3&  6&  4& 14& -3& 12\\
\end{smallarray} \right] \quad &
w &= \left[\begin{smallarray}{c}
0\\
 3\\
 3\\
 3\\
 3\\
 3\\
 6\\
 3\\
 3\\
 2
\end{smallarray}\right].
\end{align*}

When we run the semigroup trick on this linear representation, we get a DFAO
of $33$ states.  This DFAO can then be minimized, giving us the $28$-state DFAO
described in Table~\ref{enaut}.  Here $\delta(q,i)$ gives
the automaton's transition from state $q$ on input $i$,
and $\tau(q)$ gives the output associated with state $q$.

It remains to see $d(n) > 0$ for all $n$.   It is easy to see that
$0 \leq g(n) \leq 2$ and $0 \leq h(n) \leq 6$ for all $n$.
Hence for $n \geq 1$ we have
$ d(n) \leq n+1 - 3 \lfloor \log_2 (n+1) \rfloor - 3,$
which by standard estimates is positive for $n \geq 12$.
For $0 \leq n < 12$ we can check from Table~\ref{one} that $d(n) > 0$.
\begin{table}[H]
\begin{center}
\begin{tabular}{c|cccccccccccccccccccccc}
$n$ &        0& 1& 2& 3& 4& 5& 6& 7& 8& 9&10&11&12&13&14&15&16&17 \\
\hline
$r(3,A,n)$ & 1&  3&  6&  7&  9& 12& 19& 21& 24& 27& 39& 45& 52& 57& 72& 79& 87& 93 \\
\hline
$d(n)$ & 1 & 2& 3& 1& 2& 3& 7& 2& 3& 3&12& 6& 7& 5&15& 7& 8& 6\\
\end{tabular}
\end{center}
\caption{First few values of $r(3,A,n)$ and $d(n)$.}
\label{one}
\end{table}

We can sum up these calculations as follows:
\begin{theorem}
Let $F = \{ 2^{n+2} -1 \suchthat n \geq 0 \}$ and define
$A = \Enn \setminus F$.  Then the sequence
$(r(3,A,n))_{n \geq 0}$ is strictly increasing from the start,
and the difference
$d(n) = r(3,A,n)-r(3,A,n-1)$
can be calculated in time polynomial in $\log_2 n$
using the formula
$$d(n) = n+1 - 3\lfloor \log_2 (n+1) \rfloor + e(n),$$
where $e(n)$ is the sequence computed by the automaton
in Table~\ref{enaut}.
\label{main2}
\end{theorem}

\begin{remark}
It is easy to see that if $(r(k,A,n))_{n \geq 0}$ is strictly
increasing from the start, then this is also true of
$(r(k',A,n))_{n \geq 0}$ for $k' \geq k$.   So our example
of Theorem~\ref{main2} also works for all $k \geq 3$; in particular,
for $k = 4$.  Dombi \cite{Dombi:2002}
conjectured there were examples of eventually strictly
increasing sequences for $k = 4$, but was not able to prove
the existence of one.  
\end{remark}

\begin{table}[H]
\begin{center}
\begin{minipage}{3in}
\begin{center}
\begin{tabular}{c|c|c|c}
$q$ & $\delta(q,0)$ & $\delta(q,1)$ & $\tau(q)$ \\
\hline
  0& 0& 1& 0\\
  1& 2& 3& 3\\
  2& 4& 5& 3\\
  3& 6& 7& 3\\
  4& 8& 9& 3\\
  5&10&11& 3\\
  6&12&13& 6\\
  7&14& 7& 3\\
  8& 8&15& 3\\
  9&16&17& 2\\
 10&12&18&10\\
 11&19&11& 3\\
 12&12&20& 3\\
 13&21&22& 0\\
 \end{tabular}
 \end{center}
 \end{minipage}
 \begin{minipage}{3in}
 \begin{center}
 \begin{tabular}{c|c|c|c}
$q$ & $\delta(q,0)$ & $\delta(q,1)$ & $\tau(q)$ \\
\hline
 14&23&24& 9\\
 15&16&17& 0\\
 16&12&20&12\\
 17&25&17& 3\\
 18&21&22&$-1$\\
 19&23&26&13\\
 20&21&22&$-3$\\
 21&23&27& 9\\
 22&21&22& 3\\
 23&23&23& 3\\
 24&23&23& 0\\
 25&23&27&15\\
 26&23&23&$-1$\\
 27&23&23&$-3$
\end{tabular}
\end{center}
\end{minipage}
\end{center}
\caption{Automaton for $e(n)$.}
\label{enaut}
\end{table}

\end{document}